\documentclass[11pt]{article}
\usepackage{amsthm}
\usepackage{amssymb}
\usepackage{amsmath}
\usepackage{hyperref}
\usepackage{graphicx}
\usepackage{caption}
\usepackage{array}
\usepackage{enumitem}
\usepackage{mathtools}
\usepackage{etoolbox}
\usepackage[backend=biber, style=alphabetic, maxbibnames=9, maxcitenames=9, maxalphanames=9]{biblatex}
\bibliography{ConstMonodromy}
\usepackage[all,cmtip]{xy}

\pretocmd{\section}{%
  }{}{}

\numberwithin{table}{section}

\newtheorem{theorem}{Theorem}[section]
\newtheorem*{theorem*}{Theorem}
\newtheorem{proposition}[theorem]{Proposition}
\newtheorem*{proposition*}{Proposition}
\newtheorem{corollary}[theorem]{Corollary}
\newtheorem{lemma}[theorem]{Lemma}
\newtheorem*{lemma*}{Lemma}

\theoremstyle{definition}
\newtheorem{definition}[theorem]{Definition}

\newtheorem*{exercise*}{Exercise}
\newtheorem{remark}[theorem]{Remark}

\numberwithin{equation}{section}
\numberwithin{figure}{section}

\let\det\relax

\newcommand{\mc}{\mathcal}

\newcommand{\mr}{\mathrm}

\newcommand{\Z}{\mathbb{Z}}
\newcommand{\Q}{\mathbb{Q}}

\newcommand{\C}{\mathbb{C}}

\newcommand{\A}{\mathbb{A}}
\renewcommand{\O}{\mathcal{O}}

\newcommand{\Sp}{\mathrm{Sp}}

\newcommand{\sset}[2]{\lbrace{#1}\,\,|\,\,{#2}\rbrace}

\DeclareMathOperator{\det}{det}

\DeclareMathOperator{\ev}{ev}

\newcommand{\id}{\mathrm{id}}

\DeclareMathOperator{\Ind}{Ind}

\newcommand{\modulo}[1]{\,\,(\mathrm{mod}\,\,{#1})}

\DeclareMathOperator{\St}{St}

\DeclareMathOperator{\tr}{Tr}

\DeclareMathOperator{\vol}{Vol}

\begin{document}
\title{On the Bernstein--Zelevinsky classification in families}
\author{Sam Mundy}
\date{}
\maketitle

\begin{abstract}
We study the variation of admissible representations of $p$-adic $GL_n$ in families from the point of view of the Bernstein--Zelevinsky classification and show that the ramified parts of these families are rigid. We explain how to apply our results in the context of $GL_n$-eigenvarieties.
\end{abstract}

\section*{Introduction}

The $p$-adic theory of automorphic forms has, in recent decades, become an indispensable tool in studying automorphic Galois representations. One key feature of this theory is that, given a sufficiently well-behaved automorphic form, one can often deform it in a $p$-adic family of such. Such a family will be parametrized by an affinoid rigid space over $\Q_p$, often a local component of a bigger rigid analytic variety called the eigenvariety. One can then try to pass to the Galois side and obtain a family of automorphic Galois representations.

One important problem that arises in this situation is to understand how this family of Galois representations varies when restricted to the decomposition group at a finite place $v\nmid p$. This is easy, at least on a qualitative level, if the original family of automorphic forms was unramified at $v$, because then the local Galois representations at $v$ are determined by the images of Frobenius at $v$. Otherwise, if $v$ is a bad place for the family (i.e., if $v$ divides the tame level of the eigenvariety) then the variation of this family of Galois representations at $v$ can be harder to determine. Assuming we know local-global compatibility for the automorphic Galois representations in question, then on the automorphic side this comes down to studying the local admissible representations in the family at places $v$ dividing the tame level.

In this paper, in the setting of $GL_n$ automorphic forms, we show that the local variation of these admissible representations at places dividing the tame level is extremely rigid. More precisely, fix a finite collection of local nonarchimedean fields $F_1,\dotsc,F_N$ of characteristic zero, which, in the context described just above, are to be viewed as the local fields at the places $v$ dividing the tame level of the eigenvariety. We formulate a notion of family of irreducible admissible representations of $\prod_{i=1}^N GL_n(F_i)$ in Definition \ref{deffamily} which, \textit{a priori} is rather weak, but enough to capture the situation just described coming from the theory of eigenvarieties. Our main result, Theorem \ref{thmfamilythm}, when phrased in terms of the associated Weil--Deligne representations via the local Langlands correspondence, is that the restriction of such representations to the inertia group, as well as the monodromy operators, are locally constant in such families.

To proof of this result, whose main ideas we sketch now, starts by using the theory of types from the work of Schneider and Zink \cite{SZ} building on that of Bushnell--Kutzko; the idea to use types to study the local variation of automorphic representations in families is not new and appears first in the work of Bell\"iache and Chenevier \cite{BC}. The paper \cite{SZ} gives us what we need to show that the restriction of the Weil--Deligne representations associated with our family of admissible representations are constant on inertia, and that the monodromy operators can only drop rank when varying.

One of two new ingredients we introduce here to study these families is at the following step; we use Arthur--Clozel local base change \cite{AC} to change the fields $F_i$ to trivialize the supercuspidal support of the members of our family; this means now we have reduced to the case where the associated Weil--Deligne representations are trivial on inertia. The monodromy does not change, however, and we are now in a situation where our admissible representations have Iwahori level. We can then detect how big the monodromy operators are using the second new ingredient, which is Theorem \ref{thmK1fixed} below. It says the following.

Given a nonarchimedean local field $F$ of characteristic zero with ring of integers $\O_F$, fixed uniformizer $\varpi$, and residue characteristic $p$, let
\[K_{1,n}=\sset{g\in GL_n(\O_F)}{g\equiv 1\modulo{\varpi}}.\]
Let $\pi$ be an irreducible admissible representation of $GL_n(F)$ of Iwahori level with associated Weil--Deligne representation $(\rho,N)$. Write $N$ in Jordan form. Then the quotient of $p$-adic valuations
\[\frac{v_p(\dim_{\C}(\pi^{K_{1,n}}))}{v_p(\vert\O_F/\varpi\vert)}\]
is equal to the number of nonzero entries in $\exp(N)-1$.

By tracing the characteristic function of $K_{1,n}$, with $F=F_i$ for some $i$, on our base changed family of admissible representations, we show that the variation of the monodromy operators in this family is constant, which completes the proof.

Although we have just explained our main results in terms of Weil--Deligne representations, in the main body of this paper, our results will be formulated in terms of the Bernstein--Zelevinsky classification. So we recall that classification in Section \ref{secBZ}. Section \ref{secK1} is devoted to proving the aforementioned theorem on $K_{1,n}$-fixed vectors. Then we discuss families and prove our main result in Section \ref{secfamilies}.

\subsection*{Acknowledgements}
We would like to thank Eric Chen and Chris Skinner for helpful conversations. This material is based upon work supported by the National Science Foundation under Award DMS-2102473.

\subsection*{Notation and conventions}
Given a nonarchimedean local field $F$ of characteristic zero, we always write $\O_F$ for its ring of integers.

Given a standard parabolic subgroup $P$ of $GL_n$ with standard Levi $M$ and unipotent radical $N$, and a smooth admissible representation $\tau$ of $M(F)$, we always write
\[\Ind_{P(F)}^{GL_n(F)}(\tau)\]
for the unitary induction of $\tau$; thus $\tau$ is extended trivially to $N(F)$ and twisted by the square root of the modulus character of $P(F)$ before inducing.

We will have occasion in Section \ref{secK1} to consider parabolically induced representations of finite general linear groups; again these are induced by extending the inducing representation trivially over the unipotent radical, but no twist is added. The same symbol $\Ind$ will be used in this case as well.

Given several general linear groups $GL_{n_1}(F),\dotsc,GL_{n_k}(F)$ and smooth admissible representations $\pi_i$ of each $GL_{n_i}(F)$, we always write
\[\pi_1\boxtimes\dotsb\boxtimes\pi_r\]
for the exterior tensor product, considered as a representation of $GL_{n_1}(F)\times\dotsb\times GL_{n_k}(F)$.

\section{Background on the Bernstein--Zelevinsky classification}
\label{secBZ}

Fix a local nonarchimedean field $F$ of characteristic zero. We recall here the classification of irreducible admissible representations of $GL_n(F)$, $n\geq 1$, in terms of supercuspidal ones, due to Bernstein and Zelevinsky \cite{zel}. The main results of this paper are phrased in terms of this classification.

First, for any positive integer $m$, any finite collection of supercuspidal representations of $GL_m(F)$ of the form
\[\Delta=\{\tau,\tau\otimes\vert\det\vert,\dotsc,\tau\otimes\vert\det\vert^{\ell-1}\},\]
where $\ell$ is a positive integer, will be called a \textit{segment}. The integer $\ell$ is called the \textit{length} of the segment $\Delta$. Given such a segment $\Delta$, if we let $P\subset GL_{m\ell}$ be the standard parabolic subgroup with standard Levi $M=(GL_m)^{\times\ell}$, then we can form the parabolically induced representation
\[\Ind_{P(F)}^{GL_{m\ell}(F)}(\tau\boxtimes\dotsb\boxtimes(\tau\otimes\vert\det\vert^{\ell-1}))\]
(Recall our convention that this induction is normalized by the the square root of the modulus character of $P(F)$.) This induced representation has a unique irreducible quotient which we will denote by $Q(\Delta)$.

We will consider finite multisets (that is, finite sets whose elements have finite multiplicities) of segments, say $S=\{\Delta_1,\dotsc,\Delta_r\}$, where the $\Delta_i$'s, $1\leq i\leq r$ are (not necessarily distinct) segments, each consisting supercuspidal representations of $GL_{m_i}(F)$ for possibly different integers $m_i\geq 1$. In fact, for $1\leq i\leq r$, let $\ell_i$ be the length of $\Delta_i$. Let $P$ now be the standard parabolic subgroup of $GL_n$, $n=m_1\ell_1+\dotsb+m_r\ell_r$, with standard Levi subgroup $GL_{m_1\ell_1}\times\dotsb\times GL_{n_r\ell_r}$. Then according to \cite{zel}, if the segments $\Delta_1,\dotsc,\Delta_r$ are ordered in a particular way (which we specify just below) then the induced representation
\[\pi(S)=\pi(\Delta_1,\dotsc,\Delta_r)=\Ind_{P(F)}^{GL_n(F)}(Q(\Delta_1)\boxtimes\dotsb\boxtimes Q(\Delta_r))\]
has again a unique irreducible quotient, which we denote by $Q(\Delta_1,\dotsc,\Delta_r)$ or $Q(S)$. The main result of \textit{loc. cit.} is that any irreducible admissible representation can be written in this way, and moreover the (unordered) multiset $S$ is determined from $Q(S)$.

We call the collection of supercuspidal representations contained in all the $\Delta_i$'s the \textit{supercuspidal support} of $Q(S)$, or of $S$.

The order in which one must put the elements of $S$ to form $Q(S)$ is described as follows. First, given two segments $\Delta,\Delta'$, we say that $\Delta$ and $\Delta'$ are \textit{linked} if neither $\Delta$ nor $\Delta'$ is contained in the other and $\Delta\cup\Delta'$ is an interval. (This condition is empty if $\Delta$ and $\Delta'$ consist of representations of $GL_m(F)$ each for different integers $m$). Say the first element of $\Delta$ is $\tau$ and that of $\Delta'$ is $\tau'$. Then we say $\Delta$ \textit{precedes} $\Delta'$ if $\Delta$ and $\Delta'$ are linked and $\tau'=\tau\otimes\vert\det\vert^{i}$ for some integer $i\geq 1$. Then, above, we must require that the multiset $S$ is ordered so that for $1\leq i<j\leq r$, we have that $\Delta_i$ does not precede $\Delta_j$. It is always possible to order $S$ in this way.

Finally, the following partial order on multisets $S$ of segments will play a role in this paper. Let $S=\{\Delta_1,\dotsc,\Delta_r\}$ be a multiset of segments. First, assume the two segments $\Delta_i,\Delta_j$, for some $i$ and $j$ with $i\ne j$ and $1\leq i<j\leq r$, are linked. Then we may consider the segment
\[S'=\{\Delta_1,\dotsc,\widehat{\Delta}_i\dotsc,\widehat{\Delta}_j\dotsc,\Delta_r,(\Delta_i\cup\Delta_j),(\Delta_i\cap\Delta_j)\},\]
where the symbols $\widehat{\Delta}_i$ and $\widehat{\Delta}_j$ mean that $\Delta_i$ and $\Delta_j$ are omitted from $S'$, and where if $\Delta_i\cap\Delta_j$ is empty, then we omit it from $S'$ as well. We then say the segment $S'$ is \textit{obtained from} $S$ \textit{by an elementary operation}. If $S_0$ is another multiset of segments, we write $S_0<S$ if there are multisets of segments $S_1,\dotsc,S_q=S$, for some $q\geq 1$, such that each $S_{i-1}$, $1\leq i\leq q$, is obtained from $S_i$ by an elementary operation. We also write $S_0\leq S$ if $S_0<S$ or $S_0=S$. Then the relation $\leq$ defines a partial ordering on multisets of segments.

\section{A Theorem on fixed vectors by a certain compact open subgroup}
\label{secK1}

We retain the notation of the previous section. In particular we have our nonarchimedean local field $F$ of characteristic zero. Fix a uniformizer $\varpi$ in $\O_F$, and write $K=GL_n(\O_F)$.

For any positive integer $n$, we consider the compact open subgroup $K_{1,n}$ of $GL_n(F)$ defined by
\[K_{1,n}=\sset{g\in GL_n(\O_F)}{g\equiv 1\modulo{\varpi}}.\]
Then $K_{1,n}$ is normal in $GL_n(\O_F)$, and $GL_n(\O_F)/K_{1,n}$ may be naturally identified with $GL_n(\O_F/\varpi)$. As such, given a smooth admissible representation $\pi$ of $GL_n(F)$, the space $\pi^{K_{1,n}}$ of $K_{1,n}$-fixed vectors in $\pi$ is naturally a finite dimensional (by admissibility) representation of the group $GL_n(\O_F/\varpi)$.

Let $P=MN$ be a standard parabolic subgroup of $GL_n$ with Levi $M=GL_{n_1}\times\dotsb\times GL_{n_r}$, $n_1+\dotsb +n_r=n$, and unipotent radical $N$. Then
\[K_{1,n}\cap M=K_{1,n_1}\times\dotsb\times K_{1,n_r}.\]
We have the following proposition relating two parabolic inductions.

\begin{proposition}
\label{propindfixed}
Let $P=MN$ be as above, and let $\tau$ be an irreducible representation of $M(F)$. Then there is an isomorphism of $GL_n(\O_F/\varpi)$-representations
\[\Ind_{P(F)}^{GL_n(F)}(\tau)^{K_{1,n}}\cong\Ind_{P(\O_F/\varpi)}^{GL_n(\O_F/\varpi)}(\tau^{K_{1,n}\cap M(F)}).\]
\end{proposition}

The proof will require the following easy lemma.

\begin{lemma}
\label{lemindfixed}
Let $P=MN$ be as above. Let $K'\subset K$ be a normal open subgroup, and let $\tau$ be a smooth admissible representation of $M(F)$. Let $g_1,\dotsc,g_m\in GL_n(F)$ be a set of coset representatives for $P(F)\backslash GL_n(F)/K'$. Then for any $1\leq i\leq m$ and $1\leq l\leq k$, the space
\[\Ind_{P(F)}^{GL_n(F)}(\tau)^{K'},\]
contains a unique function $f_{i,l}$ such that $f_{i,l}(g_j)=\delta_{i,j}v_l$ (Kronecker delta) and such functions form a basis of this space.
\end{lemma}

\begin{proof}
Because of the Iwasawa decomposition $GL_n(F)=P(F)K$, any function 
\[f\in\Ind_{P(F)}^{GL_n(F)}(\tau)^{K'}\]
is determined by its values on the $g_i$'s. Moreover, such functions must take values in $\tau^{K'\cap P(F)}$. Thus the functions $f_{i,l}$ form a basis as long as they are well defined.

To see that the $f_{i,l}$'s are indeed well defined, let $k,k'\in K$ and $p,p'\in P(F)$ be such that $pk=p'k'$. Say $kK'=k_jK'$ and $k'K'=k_{j'}K'$. Then we claim that $j=j'$ and that $p(K'\cap P(F))=p'(K'\cap P(F))$, which implies the lemma because then
\[f_{i,l}(pk)=pf_{i,l}(k_j)=p'f_{i,l}(k_{j'})=f_{i,l}(p'k').\]
To see the claim, we note that clearly $j=j'$, whence $k_j^{-1}p^{-1}pk_j\in K'$. Thus $p^{-1}p'\in K'$ by the normality of $K'$ in $K$.
\end{proof}

\begin{proof}[Proof (of Proposition \ref{propindfixed}).]
We note that $P(F)\backslash GL_n(F)/K_{1,n}=P(\O_F)\backslash K/K_{1,n}$ by the Iwasawa decomposition, so we may choose representatives $g_1,\dotsc,g_m\in K$ for the double coset space $P(\O_F)\backslash K/K_{1,n}$ and they will represent the cosets in $P(F)\backslash GL_n(F)/K_{1,n}$. Also, since clearly we have $K_{1,n}\cap P(F)=(K_{1,n}\cap M(F))(K_{1,n}\cap N(F))$, we get that $\tau^{K_{1,n}\cap P(F)}=\tau^{K_{1,n}\cap M(F)}$. Let $v_1,\dotsc,v_k$ be a basis of $\tau^{K_{1,n}\cap M(F)}$.

We now apply the lemma with $K'=K_{1,n}$; consequently, $\Ind_{P(F)}^{GL_n(F)}(\tau)^{K_{1,n}}$ is in bijection with the space of functions $f:P(\O_F)\backslash K/K_{1,n}\to\tau^{K_{1,n}\cap M(F)}$, or what is the same, $\Ind_{P(\O_F/\varpi)}^{GL_n(\O_F/\varpi)}(\tau^{K_{1,n}\cap M(F)})$. Since the actions of $GL_n(\O_F/\varpi)$ on these spaces are defined by right translation, this gives the desired isomorphism of representations.
\end{proof}

We will also require the following lemma for the next section.

\begin{lemma}
\label{lemunrtwist}
Let $P=MN$ be as above and let $K'\subset K$ be a normal open subgroup. Let $\tau$ be a smooth admissible representation of $M(F)$ and let $\chi$ be an unramified character of $M(F)$. Then the representations of $K/K'$
\[\Ind_{P(F)}^{GL_n(F)}(\tau)^{K'}\qquad\textrm{and}\qquad\Ind_{P(F)}^{GL_n(F)}(\tau\otimes\chi)^{K'}\]
are isomorphic.
\end{lemma}

\begin{proof}
Use Lemma \ref{lemindfixed} to write down a basis of both induced representations and note that the restriction to $K$ of the action on either is the same by the unramifiedness of $\chi$.
\end{proof}

Recall that the finite group $GL_n(\O_F/\varpi)$ has a finite dimensional Steinberg representation $\St_{n,\varpi}$ (see \cite{hum} for a detailed discussion). It is a genuine representation but it may be defined virtually by the character formula
\[\St_{n,\varpi}=\sum_{P\subset GL_n}(-1)^{\vert P\vert}\Ind_{P(\O_F/\varpi)}^{GL_n(\O_F/\varpi)}(1_{P(\O_F/\varpi)}),\]
where the sum is over all $2^{n-1}$ standard parabolic subgroups $P=MN$ of $GL_n$, where $\vert P\vert$ denotes the number of simple roots in the Levi $M$ of $P$, and where $1_{P(\O_F/\varpi)}$ denotes the trivial representation of $P(\O_F/\varpi)$.

\begin{proposition}
\label{propfinst}
Let $\chi$ be an unramified character of $GL_n(F)$. Then as representations of $GL_n(\O_F/\varpi)$, we have
\[(\St_n\otimes\chi)^{K_{1,n}}\cong\St_{n,\varpi},\]
where $\St_n$ denotes the usual Steinberg representation of $GL_n(F)$.
\end{proposition}

\begin{proof}
The character formula of \cite[Proposition 9.13]{zel}, applied with the segments
\[\{\chi\vert\cdot\vert^{(1-n)/2}\},\dotsc,\{\chi\vert\cdot\vert^{(n-1)/2}\}\]
of length $1$ implies that, virtually,
\[1_{GL_n(F)}=\sum_{P}(-1)^{\vert P\vert}\iota_{P(F)}^{GL_n(F)}(\tau(P)),\]
where the sum is over all standard parabolic subgroup $P=MN$ of $GL_n$, and $\tau(P)$ is the unique irreducible quotient of
\[\iota_{(B\cap M)(F)}^{M(F)}(\chi\vert\cdot\vert^{(1-n)/2}\boxtimes\dotsb\boxtimes\chi\vert\cdot\vert^{(n-1)/2});\]
this representation $\tau(P)$ is given by an unramified twist of the Steinberg on each block of the Levi $M(F)$. Here, $B$ denotes the standard Borel subgroup of $GL_n$.

We then apply Zelevinsky's involution $(\cdot)^t$ (cf. \cite[\S 9]{zel}) to this identity, which we can do because he proves it is an algebra involution of the Grothendieck ring of smooth admissible representations of $GL_n(F)$, $n\geq 1$. Since this involution switches the Steinberg representation with the trivial representation, this gives
\[\St_n=\sum_{P}(-1)^{\vert P\vert}\iota_{P(F)}^{GL_n(F)}(1_{P(F)}).\]
Taking $K_{1,n}$ fixed vectors (which is exact) and applying Proposition \ref{propindfixed} yields the proposition at hand.
\end{proof}

\begin{corollary}
\label{corstsize}
We have
\[\dim_{\C}(\St_n^{K_{1,n}})=\vert\O_F/\varpi\vert^{n(n-1)/2}.\]
\end{corollary}

\begin{proof}
This follows from the formula
\[\vert\O_F/\varpi\vert^{n(n-1)/2}=\sum_{P}(-1)^{\vert P\vert}\vert (P\backslash GL_n)(\O_F/\varpi)\vert=\dim_{\C}(\St_{n,\varpi})\]
along with the proposition.
\end{proof}

We now prove the main result of this section.

\begin{theorem}
\label{thmK1fixed}
Let $\pi$ be an irreducible admissible representation of $GL_n(F)$ whose supercuspidal support consists of unramified characters. Write $\pi=Q(\Delta_1,\dotsc,\Delta_r)$ where the segments $\Delta_i$ consist of unramified characters. Let $\ell_i$ be the length of the segment $\Delta_i$. Then
\[\frac{v_p(\dim_{\C}(\pi^{K_{1,n}}))}{v_p(\vert\O_F/\varpi\vert)}=\sum_{i=1}^r\frac{\ell_i(\ell_i-1)}{2}.\]
\end{theorem}

\begin{proof}
Let $P$ be the standard parabolic subgroup of $GL_n$ corresponding to the partition $n=\ell_1+\dotsb+\ell_r$. Then by Propositions \ref{propindfixed} and \ref{propfinst}, we have
\[\pi(\Delta_1,\dotsc,\Delta_r)^{K_{1,n}}\cong\Ind_{P(\mc{O}/\varpi)}^{GL_n(\O_F/\varpi)}(\St_{\ell_1,\varpi}\boxtimes\dotsb\boxtimes\St_{\ell_r,\varpi}).\]
By Corollary \ref{corstsize}, the right hand side has dimension
\[\vert(P\backslash GL_n)(\O_F/\varpi)\vert\prod_{i=1}^r\vert\O_F/\varpi\vert^{\ell_i(\ell_i-1)/2}.\]
Therefore, since $p$ does not divide the order of any of the finite flag varieties $(P\backslash GL_n)(\O_F/\varpi)$ (this follows from standard formulas for the number of points of such flag varieties over finite fields) we have
\[\frac{v_p(\dim_{\C}\pi(\Delta_1,\dotsc,\Delta_r)^{K_{1,n}})}{v_p(\vert\O_F/\varpi\vert)}=\sum_{i=1}^r\frac{\ell_i(\ell_i-1)}{2}\]

Now assume that two of the segments $\Delta_j,\Delta_k$, $j\ne k$, are linked. Assume without loss of generality that $\ell_j\geq \ell_k$. Then consider the representation
\[\sigma=\pi(\Delta_1,\dotsc,\widehat{\Delta}_j,\dotsc,\widehat{\Delta}_k,\dotsc,\Delta_r,\Delta_i\cup\Delta_j,\Delta_i\cap\Delta_j),\]
where $\widehat{\Delta}_i$ means $\Delta_i$ is omitted. By the same reasoning as above, we have
\[\frac{v_p(\dim_{\C}(\sigma^{K_{1,n}}))}{v_p(\vert\O_F/\varpi\vert)}=\frac{(\ell_j+\ell_k)(\ell_j+\ell_k-1)}{2}+\frac{(\ell_j-\ell_k)(\ell_j-\ell_k-1)}{2}+\sum_{\substack{i=1\\ i\ne j,k}}^r\frac{\ell_i(\ell_i-1)}{2}.\]
Therefore
\begin{multline*}
\frac{v_p(\dim_{\C}(\sigma^{K_{1,n}}))}{v_p(\vert\O_F/\varpi\vert)}-\frac{v_p(\dim_{\C}\pi(\Delta_1,\dotsc,\Delta_r)^{K_{1,n}})}{v_p(\vert\O_F/\varpi\vert)}\\
=\frac{(\ell_j+\ell_k)(\ell_j+\ell_k-1)}{2}+\frac{(\ell_j-\ell_k)(\ell_j-\ell_k-1)}{2}-\frac{\ell_j(\ell_j-1)}{2}-\frac{\ell_k(\ell_k-1)}{2}=\ell_k>0.
\end{multline*}
It follows that if $S$ is any multiset of segments with $S<\{\Delta_1,\dotsc,\Delta_r\}$ (strict inequality) then we have
\begin{equation}
\label{eqdiffdims}
v_p(\dim_{\C}(\pi(S)^{K_{1,n}}))>v_p(\dim_{\C}\pi(\Delta_1,\dotsc,\Delta_r)^{K_{1,n}}).
\end{equation}

Now we claim that
\[v_p(\dim_{\C}(\pi(\Delta_1,\dotsc,\Delta_r)^{K_{1,n}}))=v_p(\dim_{\C}(\pi^{K_{1,n}})).\]
By the considerations above, this will finish the proof. To see the claim, we induct on the number $N$ of multisets $S$ of segments with $S<\{\Delta_1,\dotsc,\Delta_r\}$. If $N=0$ then $\pi(\Delta_1,\dotsc,\Delta_r)$ is irreducible, hence equal to $\pi$. If $N>0$, then we invoke \cite[Theorem 7.1]{zel}, which says that virtually,
\[\pi(\Delta_1,\dotsc,\Delta_r)=\pi+\sum_{S<\{\Delta_1,\dotsc,\Delta_r\}}m(S)Q(S),\]
for some are positive integers $m(S)$. Thus
\[\dim_{\C}(\pi(\Delta_1,\dotsc,\Delta_r)^{K_{1,n}})=\dim_{\C}(\pi^{K_{1,n}})+\sum_{S<\{\Delta_1,\dotsc,\Delta_r\}}m(S)\dim_{\C}(Q(S)^{K_{1,n}}).\]
By the induction hypothesis and \eqref{eqdiffdims}, we have
\[v_p(\dim_{\C}(Q(S)^{K_{1,n}}))<v_p(\dim_{\C}(\pi(\Delta_1,\dotsc,\Delta_r)^{K_{1,n}})),\]
for all $S<\{\Delta_1,\dotsc,\Delta_r\}$, which, by the strict triangle inequality, forces
\[v_p(\dim_{\C}(\pi(\Delta_1,\dotsc,\Delta_r)^{K_{1,n}}))=v_p(\dim_{\C}(\pi^{K_{1,n}})),\]
as desired. This finishes the proof.
\end{proof}

We record the following corollary of the proof which will be useful later.

\begin{corollary}
\label{corpfineq}
Let $S=\{\Delta_1,\dotsc,\Delta_r\}$ and $S'=\{\Delta_1',\dotsc,\Delta_{r'}'\}$ be multisets of segments. For $i$ with $1\leq i\leq r$ (resp. $j$ with $1\leq j\leq r'$), let $\ell_i$ (resp. $\ell_j'$) be the length of $\Delta_i$ (resp. $\Delta_j'$). Then if $S'>S$, then
\[\sum_{j=1}^{r'}\frac{\ell_j'(\ell_j'-1)}{2}>\sum_{i=1}^r\frac{\ell_i(\ell_i-1)}{2}.\]
\end{corollary}

Recall that the local Langlands correspondence of Harris--Taylor and Henniart attaches a Weil--Deligne representation $(\rho,N)$ to any irreducible admissible representation $\pi$ of $GL_n(F)$; here $\rho:W_F\to GL(V_\rho)$ is a continuous representation of the Weil group of $F$ on an $n$ dimensional $\C$-vector space $V_\rho$ and $N$ is a nilpotent ``monodromy" operator on $V_\rho$, both satisfying a certain compatibility relation. This assignment has the property that $\St_n$ is taken to a certain representation $\Sp_n=(1_n,N_{n})$ whose space $V$ has a basis $v_1,\dotsc,v_n$ such that $N_nv_i=v_{i+1}$ for $i=1,\dotsc,n-1$ and $N_nv_n=0$, and where the representation $1_n$ is the trivial representation. Moreover, if $\pi$ is supercuspidal, then the associated representation $\rho$ is irreducible when restricted to the inertia group $I_F$ and the associated operator $N$ is zero.

If $\Delta=\{\tau,\dotsc,\tau\otimes\vert\det\vert^{s-1}\}$ is a segment with $\tau$ a supercuspidal representation of $GL_n(F)$ with associated Weil--Deligne representation $(\rho,N)$, then associated with $Q(\Delta)$ is the representation $(\rho\otimes 1_n,\id\otimes N_n)$. Further, if $\Delta_1,\dotsc,\Delta_r$ are segments, then $Q(\Delta_1,\dotsc,\Delta_r)$ has associated representation 
\[(\rho_1\oplus\dotsb\oplus\rho_r,N_1\oplus\dotsc\oplus N_r),\]
where $(\rho_i,N_i)$ is the Weil--Deligne representation attached to $Q(\Delta_i)$.

The following is thus way to rephrase the above theorem in terms of Weil--Deligne representations.

\begin{corollary}
Let $\pi$ be an irreducible admissible representation of $GL_n$ with associated Weil--Deligne representation $(\rho,N)$. Assume $\rho|_{I_F}$ is trivial. Let $v_1,\dotsc,v_n$ be a basis of the space $V$ of $(\rho,N)$ in which $N$ is written in Jordan form. Then
\[\frac{v_p(\dim_{\C}(\pi^{K_{1,n}}))}{v_p(\vert\O_F/\varpi\vert)}\]
is the number of nonzero entries of the matrix of $\exp(N)-\id_V$ in the basis $v_1,\dotsc,v_n$.
\end{corollary}

\section{Application to families of admissible representations}
\label{secfamilies}

We now consider the variation of admissible representations in certain families; these families parametrize Hecke traces and they arise in the context of eigenvarieties as we explain below.

Fix throughout this section a finite collection of (not necessarily distinct) nonarchimedean local fields of characteristic zero $F_1,\dotsc,F_N$. Fix also an auxiliary field of characteristic zero $L$ with an embedding $L\hookrightarrow\C$. For each $i=1,\dotsc,N$, we will consider the Hecke algebras $\mc{H}_i=C_c^\infty(GL_n(F_i),L)$. Then given a smooth admissible representation $\pi_i$ of $GL_n(F_i)$, the algebra $\mc{H}_i$ acts on $\pi_i$ by convolution, and given any $f_i\in\mc{H}_i$, we can consider the trace $\tr(f_i|\pi_i)\in\C$. The representation $\pi_i$ is determined by this trace map.

We also consider the algebra
\[\mc{H}=\mc{H}_1\otimes_L\dotsb\otimes_L\mc{H}_N,\]
which acts on exterior tensor products of the form
\[\pi_1\boxtimes\dotsb\boxtimes\pi_N,\]
for any smooth admissible representations $\pi_i$ of $GL_n(F_i)$, $i=1,\dotsc,N$.

We also fix throughout this section a topological space $X$ and a dense subset $\Sigma\subset X$, and we let $R$ be an $L$-subalgebra of the ring of all functions $X\to\C$ such that, for any $\phi\in R$ and any $a\in\C$, $\phi^{-1}(a)$ is a closed subset of $X$ (i.e., we require $\phi$ to be continuous with respect to the cofinite topology on $\C$).

\begin{definition}
\label{deffamily}
By a $(\Sigma,R)$\textit{-family of irreducible admissible representations of} $\prod_i GL_n(F_i)$, we mean an $L$-linear map $T:\mc{H}\to R$ with the following property: For any $x\in\Sigma$, the composition $T_x=\ev_x\circ T$ of $T$ with the evaluation map $\ev_x$ at $x$ is equal to the trace on an irreducible admissible representation of $\prod_i GL_n(F_i)$. In other words, for any $x\in\Sigma$, we require that there exists irreducible admissible representations $\pi_{i,x}$ of $GL_n(F_i)$, $i=1,\dotsc,N$, such that for any $f\in\mc{H}$, we have
\[T_x(f)=\tr(f|\pi_{1,x}\boxtimes\dotsb\boxtimes\pi_{N,x}).\]
\end{definition}

The main result of this section is the following.

\begin{theorem}
\label{thmfamilythm}
Let $T$ be a $(\Sigma,R)$-family of irreducible admissible representations of $\prod_i GL_n(F_i)$. For $x\in\Sigma$, let $\pi_{i,x}$ be irreducible admissible representations of $GL_n(F_i)$, $i=1,\dotsc,N$, such that
\[T_x(f)=\tr(f|\pi_{1,x}\boxtimes\dotsb\boxtimes\pi_{N,x}).\]
Fix $x_0\in\Sigma$ and write
\[\pi_{i,x_0}\cong Q(\Delta_{1,i,x_0},\dotsc,\Delta_{r_0,i,x_0})\]
in the Bernstein--Zelevinsky classification, for some segments $\Delta_{1,i,x_0},\dotsc,\Delta_{r_0,i,x_0}$. Then there is an open and closed neighborhood $X_0$ of $x_0$ in $X$ with the following property: For any $x\in X_0\cap\Sigma$, there are unramified characters $\chi_{k,i,x}$ of $F_i^\times$, $1\leq i\leq N$, $1\leq k\leq r_0$, such that
\[\pi_{i,x}\cong Q(\Delta_{1,i,x_0}\otimes\chi_{1,i,x},\dotsc,\Delta_{r_0,i,x_0}\otimes\chi_{r,i,x}).\]
Here, for any $k$ with $1\leq k\leq r_0$, the segment $\Delta_{k,i,x}\otimes\chi_{k,i,x}$ denotes the segment whose members are just those of $\Delta_{k,i,x}$ each twisted by $\chi_{k,i,x}$.
\end{theorem}

\begin{remark}
In terms of Weil--Deligne representations, the conclusion of the theorem is the following. For any $x\in\Sigma_0$, let $(\rho_{i,x},N_{i,x})$ be the Weil--Deligne representation attached to $\pi_{i,x}$, $1\leq i\leq N$. Then the pairs $(\rho_{i,x}|_{I_{F_i}},N_{i,x})$ are constant on $X_0\cap\Sigma$; only the Frobenius action changes.
\end{remark}

The proof of Theorem \ref{thmfamilythm} has three main ingredients: There is the paper of Schneider and Zink \cite{SZ} where they construct certain types in the sense of Bushnell--Kutzko, there is the Arthur--Clozel \cite{AC} local base change for $GL_n$, and there are the main results of the previous section.

\begin{proof}[Proof (of Theorem \ref{thmfamilythm})]
The proof will proceed in two separate steps. In Step 1 we find an open and closed subset $X_0\subset X$ such that for every $x\in X_0\cap\Sigma$, we have $\pi_{i,x}\cong Q(S_{i,x})$ for some multiset of segments $S_{i,x}$ with $S_{i,x}\geq\{\Delta_{1,i,x_0}\otimes\chi_{1,i},\dotsc,\Delta_{r_0,i,x_0}\otimes\chi_{N,i}\}$ for some unramified characters $\chi_{1,i},\dotsc,\chi_{N,i}$ of $F_i^\times$. In Step 2, we show the reverse inequality (after possibly shrinking $X_0$).

\textit{Step 1.} First, for any $i$ with $1\leq i\leq N$, \cite[Proposition 6.2 i]{SZ} implies the existence of an idempotent $e_i\in\mc{H}_i$ such that
\[e_i Q(\Delta_{1,i,x_0}\otimes\chi_{1,i},\dotsc,\Delta_{r_0,i,x_0}\otimes\chi_{r_0,i})\ne 0,\]
for any unramified characters $\chi_{1,i},\dotsc,\chi_{r_0,i}$ of $F_i^\times$. Moreover, if $S$ is a multiset of segments with the same supercuspidal support as $\{\Delta_{1,i,x_0},\dotsc,\Delta_{r_0,i,x_0}\}$, then \cite[Proposition 6.2 ii]{SZ} implies that if $e_i Q(S)\ne 0$, then there are unramified characters $\chi_{1,i},\dotsc,\chi_{r_0,i}$ of $F_i^\times$ such that $S\geq\{\Delta_{1,i,x_0}\otimes\chi_{1,i},\dotsc,\Delta_{r_0,i,x_0}\otimes\chi_{r_0,i}\}$. Finally, if instead $S$ is a multiset of segments with supercuspidal support different from that of $\{\Delta_{1,i,x_0},\dotsc,\Delta_{r_0,i,x_0}\}$ up to unramified twist, then $e_iQ(S)=0$.

Now say $S$ and $S'$ are multisets of segments with the same supercuspidal support as $\{\Delta_{1,i,x_0},\dotsc,\Delta_{r_0,i,x_0}\}$ up to unramified twist, and that $S'$ is obtained from $S$ by twisting each of its segments by an unramified character. Then since the idempotent $e_i$ comes from an irreducible smooth representation of an open subgroup of $GL_n(\O_{F_i})$, by Lemma \ref{lemunrtwist}, the spaces $e_i\pi(S)$ and $e_i\pi(S')$ give the same representations of $GL_n(\O_{F_i})$. In particular, since there are only finitely many such multisets $S$ up to unramified twist, there can only be finitely many possible values for the quantities
\[\tr(e_i|Q(S))=\dim_\C(e_iQ(S)).\]

Now consider the function $\phi=T(e_1\otimes\dotsb\otimes e_N)\in R$. Then $\phi(x_0)\ne 0$ by construction, and for any $x\in\Sigma$, we have $\phi(x)$ is among a finite list of nonnegative integers. Therefore, it follows from our definition of family and the density of $\Sigma$ that $\phi^{-1}(\phi(x_0))$ and its complement are closed in $X$, and therefore $\phi$ is constant on some open and closed subset $X_0\subset X$ containing $x_0$. Thus for every $x\in  X_0\cap\Sigma$, we must have $\pi_{i,x}\cong Q(S_{i,x})$ for some multiset of segments $S_{i,x}$ with $S_{i,x}\geq\{\Delta_{1,i,x_0}\otimes\chi_{1,i},\dotsc,\Delta_{r_0,i,x_0}\otimes\chi_{N,i}\}$ for some unramified characters $\chi_{1,i},\dotsc,\chi_{N,i}$ of $F_i^\times$. This concludes Step 1.

\textit{Step 2.} We must now prove the reverse inequality. To do this, for each $i$ with $1\leq i\leq N$, let $F_i'$ be a finite Galois extension of $F_i$ with the following property: For all $x\in X_0\cap\Sigma$, if $(\rho_{i,x},N_{i,x})$ is the Weil--Deligne representation attached to $\pi_{i,x}$, then $\rho_{i,x}|_{F_i'}$ is trivial. Such an $F_i'$ exists because we just showed that, up to unramified twist, the supercuspidal support of $\pi_{i,x}$ is independent of $x\in X_0\cap\Sigma$.

Since finite Galois extensions of local fields are solvable, we may invoke the results of \cite[Chapter 1]{AC}, which give us representations $\pi_{i,x}'$ of $GL_n(F_i')$ which are the base change lifts of $\pi_{i,x}$ for any $x\in\Sigma$. By our choice of $F_i'$, the representations $\pi_{i,x}'$ for $x\in X_0\cap\Sigma$ may be described as follows. For each $k$ with $1\leq k\leq r$, let $\tau_{k,i,x}$ be the first supercuspidal representation in the segment $\Delta_{k,i,x}$. Say it is a representation of $GL_{n_k}(F_i)$. Then the base change $\tau_{k,i,x}'$ of $\tau_{k,i,x}$ to $F_i'$ is an unramified representation and therefore occurs as a constituent in the Borel induction of several unramified characters $\psi_{1,k,i,x},\dotsc,\psi_{n_k,k,i,x}$ of $(F_i')^\times$. For any $m$ with $1\leq m\leq n_k$, let $\Delta_{m,k,i,x}'$ be the segment of length equal to that of $\Delta_{k,i,x}$ and starting with the character $\psi_{m,k,i,x}$. Let $S_{i,x}'$ be the multiset of segments consisting of the segments $\Delta_{m,k,i,x}'$ for $1\leq k\leq r$ and $1\leq m\leq n_k$. Then $\pi_{i,x}'\cong Q(S_{i,x}')$.

Now for any $i$ with $1\leq i\leq N$, let $\mc{H}_i'=C_c^\infty(GL_n(F_i'),L)$, and let
\[\mc{H}'=\mc{H}_1'\otimes_L\dotsb\otimes_L\mc{H}_N'.\]
Then \cite[Chapter 1]{AC} also implies the following. Given any $f_i'\in\mc{H}_i'$, there is a transferred operator $f_i\in\mc{H}_i$ such that for any irreducible admissible representation $\sigma$ of $GL_n(F_i)$ with base change $\sigma'$ to $F_i'$, we have
\[\tr(f_i'|\sigma')=\tr(f_i|\sigma).\]
Thus we can define a map $T':\mc{H}'\to R$ by
\[T'(f_1'\otimes\dotsb\otimes f_N')=T(f_1\otimes\dotsb\otimes f_N),\]
where each $f_i$ is associated with $f_i'$ as just described. One checks easily that $T'$ is a $(\Sigma,R)$-family of irreducible admissible representations of $\prod_i GL_n(F_i')$ and that, furthermore, for any $x\in\Sigma$, we have
\[T_x'(f')=\tr(f'|\pi_{1,x}'\boxtimes\dotsb\boxtimes\pi_{N,x}').\]

Now fix $j$ with $1\leq j\leq N$. Let $x\in X_0\cap\Sigma$. By construction, for any $i$ with $1\leq i\leq N$, the representations $\pi_{i,x}'$ are constituents of unramified principal series, and therefore possess fixed vectors by the Iwahori subgroup $I_i'$ defined by
\[I_i'=\sset{g\in GL_n(\O_{F_i'})}{(g\textrm{ mod }\varpi_i')\in B(\O_{F_i'}/\varpi_i')},\]
where $\varpi_i'$ denotes a fixed uniformizer in $\O_{F_i'}$. Let $f_i'=\vol(I_i')^{-1}1_{I_i'}$ if $i\ne j$ and let $f_j'=\vol(K_{1,i,n}')^{-1}1_{K_{1,i,n}'}$, where
\[K_{1,i,n}'=\sset{g\in GL_n(\O_{F_i'})}{g\equiv 1\modulo{\varpi_i'}},\]
similarly to the previous section. Then set $f'=f_1'\otimes\dotsb\otimes f_N'$. Then we have
\[T_x'(f')=\dim_{\C}((\pi_{j,x}')^{K_{1,i,n}'})\prod_{i\ne j}\dim_{\C}((\pi_{i,x}')^{I_i'}).\]
Note that there are only finitely many possible values for the above expression, since
\[\dim_{\C}((\pi_{i,x}')^{I_i'})\leq n!\]
and
\[\dim_{\C}((\pi_{j,x}')^{K_{1,i,n}'})\leq \vert GL_n(\O_{F_i'}/\varpi_i')\vert.\]

Let $p_i$ denote the residue characteristic of $F_i$. We would like at this point to apply Theorem \ref{thmK1fixed} and Corollary \ref{corpfineq} to recover the lengths of the intervals comprising $S_{i,x}'$, but unfortunately the $p_j$-adic valuation of any of the terms $\dim_{\C}((\pi_{i,x}')^{I_i'})$ above could be nonzero. To overcome this, we make a second base change. Let $F_i''=F_i'$ if $i\ne j$, and let $F_j''$ be the unramified quadratic extension of $F_j'$. Then we get base changed representations $\pi_{i,x}''$ for any $i$ (which are just the representations $\pi_{i,x}'$ if $i\ne j$) and analogously defined Hecke algebras $\mc{H}_i''$ and $\mc{H}''$, and also an analogously defined family $T''$. We consider the operators $f_i''$ defined by $f_i''=f_i'$ if $i\ne j$ and $f_j''=\vol(K_{1,i,n}'')^{-1}1_{K_{1,i,n}''}$, where $K_{1,i,n}''$ is the analogously defined subgroup of $GL_n(\O_{F_i''})$. Finally, let $f''=f_1''\otimes\dotsb\otimes f_N''$. We then have
\[T_x''(f'')=\dim_{\C}((\pi_{j,x}'')^{K_{1,i,n}''})\prod_{i\ne j}\dim_{\C}((\pi_{i,x}')^{I_i'}),\]
for any $x\in X_0\cap\Sigma$

Now similarly to the conclusion of Step 1, after possibly shrinking $X_0$, we find that the value $T_x''(f'')/T_x'(f')$ is constant on $X_0$, and by Theorem \ref{thmK1fixed}, we have
\begin{equation}
\label{eqvpj}
v_{p_j}(T_x''(f'')/T_x'(f'))=\sum_{k=1}^{r_x}n_k\frac{\ell_{k,j,x}(\ell_{k,j,x}-1)}{2},
\end{equation}
where $\ell_{k,j,x}$ is the length of the interval $\Delta_{k,j,x}$.

Now assume that $S_{j,x}>\{\Delta_{1,j,x_0}\otimes\chi_{1,j},\dotsc,\Delta_{r,j,x_0}\otimes\chi_{r_0,j}\}$ (strict inequality) for some unramified characters $\chi_{1,i},\dotsc,\chi_{r_0,i}$ of $F_i^\times$. It is not too difficult to see from the description we gave above that $S_{j,x}'$ is strictly greater than the multiset consisting of the segments $\Delta_{m,k,j,x_0}'\otimes\chi_{k,j}'$, $1\leq k\leq r_0$, $1\leq m\leq n_k$, where $\chi_{k,j}'$ is the base change of $\chi_{k,j}$ to $(F_j')^\times$; moreover a similar statement holds for the base change to $F_j''$. By Corollary \ref{corpfineq}, we have
\[\sum_{k=1}^{r_x}n_{k,x}\frac{\ell_{i,x}(\ell_{i,x}-1)}{2}>\sum_{k=1}^{r_0}n_{k,x_0}\frac{\ell_{i,x_0}(\ell_{i,x_0}-1)}{2},\]
which contradicts the constancy of \eqref{eqvpj} on $X_0\cap\Sigma$. This contradiction completes Step 2, hence also the proof of the theorem.
\end{proof}

We finish by singling out a source of examples of the $(\Sigma,R)$-families just studied. These come from the eigenvarieties of the type considered by Ash--Stevens \cite{AS} and Urban \cite{urbanev}. Let $p$ be a rational prime and $E$ a totally real number field. Fix an isomorphism $\C\cong\overline{\Q}_p$. Let $S$ be a set of finite places of $E$ including all of those above $p$. Consider the spherical Hecke algebra $\mc{H}_{\mr{sph}}^{S}$ with coefficients in $\Q_p$ given by
\[\mc{H}_{\mr{sph}}^{S}=C_c^\infty({\textstyle\prod_{v\notin S}}GL_n(\O_{E_v})\backslash GL_n(\A_E^S)/{\textstyle\prod_{v\notin S}}GL_n(\O_{E_v}),\Q_p),\]
where $\A_E^S$ denotes the adeles of $E$ away from $S$. Let
\[\mc{H}_S^p=C_c^\infty({\textstyle\prod_{v\in S,\,v\nmid p}}GL_n(E_v),\Q_p),\]
which is the full Hecke algebra over $\Q_p$ at the places in $S$ not dividing $p$.

Then, using the theory of eigenvarieties, one can often construct affinoid rigid analytic spaces $\mc{E}$ over $\Q_p$ along with the following. First, there is a $\Z_p$-algebra $\mc{U}_p$ of $U_p$-operators at the places above $p$ (see Urban \cite[\S 4.1.1]{urbanev}) and a $\Q_p$-linear map
\[I:\mc{H}_S^p\otimes_{\Q_p}\mc{H}_{\mr{sph}}^{S}\otimes_{\Z_p}\mc{U}_p\to\mc{O}(\mc{E}),\]
where $\mc{O}(\mc{E})$ is the ring of analytic functions on $\mc{E}$. The map $I$ has the property that for any $x\in\mc{E}(\overline{\Q}_p)$, the specialization $I_x$ of $I$ at $x$ (that is, the composition of the $\overline\Q_p$-point $x$ with $I$) is the trace of a $p$-stabilization (see Urban \cite[\S 4.1.9]{urbanev}) of a smooth admissible representation of $GL_n(\A_E)$; moreover this $p$-stabilization is irreducible outside a Zariski closed subset of $\mc{E}$, and is there is a Zariski dense subset of points $x$ in $\mc{E}(\overline{\Q}_p)$ where $I_x$ is the $p$-stabilization of an automorphic representation of $GL_n(\A_E)$.

Now view $\mc{E}(\overline\Q_p)$ as a topological space with a Zariski topology whose closed subsets are given by the $\overline{\Q}_p$-points of any closed affinoid subset of $\mc{E}$. Then functions in $\O(\mc{E})$ are naturally viewed as functions from $\mc{E}(\overline\Q_p)$ to $\overline{\Q}_p$.

The Hecke algebras $\mc{H}_{\mr{sph}}^{S}$ and $\mc{U}_p$ have identity elements. We may therefore consider the map
\[T:\mc{H}_S^p\to\overline\Q_p,\qquad T(f)=I(f\otimes 1\otimes 1).\]
Let $\Sigma$ be the subset of points $x$ in $\mc{E}(\overline\Q_p)$ where $I_x$ is the trace of an irreducible $p$-stabilized representation. Then $T$ is a $(\Sigma,\O(\mc{E}))$-family of irreducible admissible representations of $\prod_{v\in S,\,v\nmid p}GL_n(E_v)$; here in this example the field $L$ from above is $\Q_p$. The map $T$ is indeed such a family since for any analytic function $\phi\in\O(\mc{E})$, the zero locus of $\phi$ is a closed subvariety of $\mc{E}$.

Our Theorem \ref{thmfamilythm} then tells us how the local components at bad places of the $p$-adic automorphic representations parametrized by $\mc{E}$ vary.

Strictly speaking, the eigenvarieties of Urban do not include this example if $n>2$, because $GL_{n/E}$ then does not have discrete series at infinity. But there are variants, such as the twisted eigenvarieties of Xiang \cite{xiang}, which provide examples in this case. Alternatively, Urban's eigenvarieties do provide examples of these families in the case of unitary groups defined over $E$ when all the bad places are split in the CM field used to define the unitary group.

\printbibliography
\end{document}